\documentclass[12pt,reqno]{amsart}

\usepackage{fullpage}
\usepackage{amsmath}
\usepackage{amsfonts}
\usepackage{amssymb}
\usepackage{amsthm}
\usepackage{color}
\usepackage{hyperref}

\usepackage{amscd}
\usepackage{mathrsfs}

\numberwithin{equation}{section}

\newcommand{\beq}{\begin{equation}}
\newcommand{\eeq}{\end{equation}}

\newcommand\be{\begin{equation}}
\newcommand\ee{\end{equation}}
\newcommand\bea{\begin{eqnarray}}
\newcommand\eea{\end{eqnarray}}

\newcommand\bel{\begin{equation}}
\newcommand\eel{\end{equation}}
\newcommand\beal{\begin{eqnarray}}
\newcommand\eeal{\end{eqnarray}}

\newcommand\bi{\begin{itemize}}
\newcommand\ei{\end{itemize}}
\newcommand\ben{\begin{enumerate}}
\newcommand\een{\end{enumerate}}

\newcommand{\twocase}[5]{#1 \begin{cases} #2 & \text{{\rm #3}}\\ #4 &\text{{\rm #5}} \end{cases}   }
\newcommand{\ncr}[2]{{#1 \choose #2}}
\global\long\def\ncr#1#2{{#1  \choose #2}}

\newcommand{\E}{\mathbb{E}}

\newtheorem{thm}{Theorem}[section]

\newtheorem{lem}[thm]{Lemma}
\newtheorem{rek}[thm]{Remark}

\newtheorem{theorem}{Theorem}[section]

\newtheorem{definition}[theorem]{Definition}

\numberwithin{equation}{section}

\newtheorem*{remark*}{Remark}
\numberwithin{remark}{section}
\numberwithin{subsubcase}{subcase}
\numberwithin{subsubsection}{subsection}

\begin{document}

\title{When Generalized Sumsets are Difference Dominated}

\author{Virginia Hogan}
\email{ginny6@stanford.edu}
\address{Department of Mathematics, Stanford University, Stanford, CA 94305}

\author{Steven J. Miller}
\email{sjm1@williams.edu, Steven.Miller.MC.96@aya.yale.edu} \address{Department of Mathematics and Statistics, Williams College, Williamstown, MA 01267}

\subjclass[2010]{11P99 (primary), 11K99 (secondary).}

\keywords{Sum dominated sets, more sum than difference sets, strong concentration, phase transitions.}

\date{\today}

\thanks{The first named author was partially supported by NSF Grant DMS0850577, and second named author was partially supported by NSF Grant DMS0970067. We thank the participants of the 2012 SMALL REU program, especially Kevin Vissuet, as well as Kevin O'Bryant and Dmitrii Zhelezov for helpful discussions.}

\begin{abstract}
We study the relationship between the number of minus signs in a generalized sumset, $A+\cdots+A-\cdots-A$, and its cardinality; without loss of generality we may assume there are at least as many positive signs as negative signs. As addition is commutative and subtraction is not, we expect that for most $A$ a combination with more minus signs has more elements than one with fewer; however, recently Iyer, Lazarev, Miller and Zhang \cite{ILMZ} proved that a positive percentage of the time the combination with fewer minus signs can have more elements. Their analysis involves choosing sets $A$ uniformly at random from $\{0,\dots,N\}$; this is equivalent to independently choosing each element of $\{0,\dots,N\}$ to be in $A$ with probability $1/2$. We investigate what happens when instead each element is chosen with probability $p(N)$, with $\lim_{N\to\infty} p(N) =0$. We prove that the set with more minus signs is larger with probability $1$ as $N\to\infty$ if $p(N)=cN^{-\delta}$ for $\delta\ge\frac{h-1}{h}$, where $h$ is the number of total summands in $A+\cdots+A-\cdots-A$, and explicitly quantify their relative sizes. The results generalize earlier work of Hegarty and Miller \cite{HM}, and we see a phase transition in the behavior of the cardinalities when $\delta = \frac{h-1}{h}$.
\end{abstract}

\maketitle

\tableofcontents

\section{Introduction}\label{sec:intro}

\subsection{Previous Results}


Let $A$ be a subset of the integers. We define the \textbf{sumset} $A+A$ and the \textbf{difference set} $A-A$ by
\be
A+A\ =\ \{a_{1}+a_{2}:a_{i}\in A\},\ \ \ A-A\ = \ \{a_{1}-a_{2}:a_{i}\in A\}.
\ee Many important problems in number theory are related to these sets and their generalizations. For example, if $P$ denotes the set of primes and $K$ the set of $k$\textsuperscript{th} powers of positive integers, then the Goldbach conjecture is equivalent to $P+P$ contains all even numbers, the twin prime conjecture is $P-P$ contains 2 infinitely often, Fermat's Last Theorem is $(K+K) \cap K$ is empty if $k \ge 3$, and Waring's problem is that for each $k$ there is an $s$ such that $K + \cdots + K$ ($s$ times) contains all positive integers.

Note the last problem involves more than one binary operation; the main goal of this paper is to explore what happens to generalized sumsets in different models. Before stating our results, we review some previous work. As addition is commutative and subtraction is not, a typical pair of integers generates two differences but only one sum. It is therefore reasonable to expect a generic finite set $A$ has a larger
difference set than sumset. If this is the case then we say $A$ is \textbf{difference dominated}, while if the two sets have the same size we say the set is \textbf{balanced}, and if the sumset is larger then $A$ is \textbf{sum dominated} (also called a \textbf{more sums than differences (MSTD) set}). It was conjectured that if $A$ is chosen uniformly at random from $\{0,\dots,N\}$ then as $N\to\infty$ almost all sets are difference dominated. In 2007, however, Martin and O'Bryant \cite{MO} disproved this conjecture by showing a positive percentage of sets are sum dominated. The percentage is small, around $4.5 \cdot 10^{-4}$ \cite{Zh}.

While these results imply that sum dominated sets are not too rare, this is a consequence of how the sets are chosen. An equivalent formulation is that each element of $I_N := \{0,\dots,N\}$ is chosen to be in $A$ with probability 1/2. With high probability
a randomly chosen subset $A$ has approximately $N/2$ elements (with errors of size $\sqrt{N}$). Thus the density of a generic subset to the underlying set $I_{N}$ is quite high, typically about $1/2$. Because it is so high, when we look at the sumset (resp., difference set) of a typical $A$ there are many ways of expressing elements as a sum (resp., difference) of two elements of $A$. Almost all possible sums and differences are realized; the expected number of missing differences is 6, while the expected number of missing sums is 10. Thus, a typical set needs just a small nudge to become sum dominated. This can be accomplished by appropriately choosing the fringe elements of $A$ (the elements near 0 and $N$), as almost surely changes at the fringes do not affect whether or not most possible sums and differences are realized.

This observation suggests that instead of taking each element of $I_N$ with probability 1/2 (or any fixed, non-zero probability), we should instead explore what happens when all of these elements are chosen independently with probability $p(N)$, where $p$ is some function tending to zero; this is a binomial model with parameter $p(N)$. Such an analysis was done by Hegarty and Miller \cite{HM} in 2009. They showed that if $p(N) = c N^{-\delta}$ for some $\delta \in (0,1)$, then almost surely $A$ is difference dominated. The analysis breaks into three cases based on the probability for choosing elements in $A$. The authors study \textbf{fast decay} ($\delta>1/2$), \textbf{critical decay} ($\delta=1/2$), and \textbf{slow decay} ($\delta<1/2$). There is a phase transition at $\delta = 1/2$, leading to the name critical decay.

Before stating their results we first introduce some definitions, notation, conventions, and standard facts that we use in our results as well.

We start with notation for sizes. By $f(x)=O(g(x))$ we mean that there exist constants $x_{0}$ and $C$ such that for all $x\ge x_{0}$, $|f(x)|\le Cg(x)$. We write $f(x)=\Theta(g(x))$ if both $f(x)=O(g(x))$ and $g(x)=O(f(x))$. If $\lim_{x\to\infty}f(x)/g(x)=0$ then we write $f(x)=o(g(x))$, which is equivalent to $f(x) \ll g(x)$.

As the fundamental objects of study are sizes of sets, we need a way to denote asymptotic behavior. Let $X$ be a real-valued random variable depending on some positive integer parameter $N$, and let $f(N)$ be some real-valued function. By ``$X\sim f(N)$'' we mean that, for any $\epsilon_{1},\epsilon_{2}>0$, there exists $N_{\epsilon_{1},\epsilon_{2}}>0$ such that, for all $N>N_{\epsilon_{1},\epsilon_{2}}$,
\be P(X\notin\left[(1-\epsilon_{1})f(N),(1+\epsilon_{1})(f(N)\right])\ <\ \epsilon_{2}. \ee

We now state the main past result, which we will generalize.

\begin{thm}[Hegarty-Miller \cite{HM}]\label{thm:mainuniform}
Let $p : \mathbb{N} \rightarrow (0,1)$ be any function such that
\be\label{eq:old13} N^{-1}\ =\ o(p(N)) \ \ \ \ {\rm and}\ \ \ \
p(N)\ =\ o(1).\ee For each $N \in \mathbb{N}$ let $A$ be a random
subset of $I_{N}$ chosen according to a binomial distribution with
parameter $p(N)$. Then, as $N \rightarrow \infty$, the probability
that $A$ is difference dominated tends to one.
\par More precisely, let $\mathscr{S},
\mathscr{D}$ denote respectively the random variables $|A+A|$ and $|A-A|$.
Then the following three situations arise:
\\
\\
(i) $p(N) = o(N^{-1/2})$ : Then \be\label{eq:old14} \mathscr{S}\ \sim\
{(N\cdot p(N))^{2} \over 2} \;\;\; {\hbox{and}} \;\;\; \mathscr{D}
\sim 2\mathscr{S}\ \sim\ (N \cdot p(N))^{2}. \ee (ii) $p(N) = c \cdot
N^{-1/2}$ for some $c \in (0,\infty)$ : Define the function $g :
(0,\infty) \rightarrow (0,2)$ by \be g(x)\ :=\ 2\left(\frac{e^{-x} -
(1-x)}{x}\right). \ee Then \be\label{eq:old16} \mathscr{S}\ \sim\
g\left({c^{2} \over 2}\right) N \;\;\; {\hbox{and}} \;\;\;
\mathscr{D}\ \sim\ g(c^{2}) N. \ee (iii) $N^{-1/2} = o(p(N))$ : Let
$\mathscr{S}^{c} := (2N+1) - \mathscr{S}$, $\mathscr{D}^{c} :=
(2N+1) - \mathscr{D}$. Then \be \mathscr{S}^{c}\ \sim\ 2 \cdot
\mathscr{D}^{c}\ \sim\ {4 \over p(N)^{2}}. \ee
\end{thm}

Notice there is a phase transition at $\delta = 1/2$, where $|A-A|$ goes from almost surely having twice as many elements as $|A+A|$ (when $\delta > 1/2$) to having the same number of elements to first order (when $\delta < 1/2$); further, an explicit, tractable formula is obtained for the relative sizes when $\delta = 1/2$ as a simple function of $c$.

The goal of this paper is to generalize this theorem to arbitrary combinations of sums and differences.

\subsection{Results}

Before stating our results, we need some combinatorial results. We use the extended definition of the binomial coefficient, setting $\ncr ab=0$ for integers $0 \le a<b$. A central result, which we use again and again, is the stars and bars (or cookie) problem: for any pair of positive integers $n,k$, the number of distinct $k$-tuples of non-negative integers that sum to $n$ is $\ncr{n+k-1}{k-1}$. Note this is equivalent to counting the number of solutions in non-negative integers to $x_1 + \cdots + x_k = n$. This is readily found. If we choose $k-1$ objects from $n+k-1$ (there are $\ncr{n+k-1}{k-1}$ ways to do so), we partition the remaining $n$ objects into $k$ sets, and there is a one-to-one correspondence between these partitions and our desired solutions.

In our investigations below we always choose elements for our set $A$ from $I_N := \{0, \dots, N\}$ independently with probability $p(N)=cN^{-\delta}$ for fixed $\delta\in\left(0,1\right)$ and $c>0$.

\begin{itemize}
\item Given a set $A$ we define its generalized sumset $A_{s,d}$ with $s$ sums and $d$ differences to be $A+\cdots+A-\cdots-A$; as we are only interested in cardinalities we may always assume $d \le s$.
\item We write $|A_{s,d}|$ for its size. We always use $h$ for the number of summands, so $h =s+d$. \item An $h_{(s,d)}$-tuple is a set of $h=s+d$ integers, $\left\{a_{1},\ldots, a_{s},a_{s+1},\ldots,a_{h}\right\}$.
\item If the associated sum $\sum_{i=1}^s a_i - \sum_{j=1}^d a_{s+j}$ equals $\lambda$ then we say the tuple generates $\lambda$. Note that the generalized sumset is the set of all numbers generated by $h_{(s,d)}$-tuples of elements of $A$.
\item Related to this is $R(n,s,d)$, which we define to be the number of ways to generate $n$ through $h_{(s,d)}$-tuples of integers drawn from $\{0, \dots, N\}$. As $R(n,s,d)$ counts all permutations equally, order matters; for example, if $a_{1}+a_{2}-a_{3}=n$, then $R(n,2,1)$ counts $(a_{1},a_{2},a_{3})$ and $(a_{2},a_{3},a_{1})$ as two different entities. As $N$ is fixed throughout our calculations and then sent to infinity only at the end, to simplify notation we write $R(n,s,d)$, though really it should be $R_N(n,s,d)$ to emphasize this dependence.
\end{itemize}


In the course of our investigations we encounter the following constants and functions. For $k$ a positive integer and $j\in(0,h/2)$, set
\begin{equation}\label{eq:eqnB}
b_{h,k} \ := \ \frac{1}{k!(h-1)!^{k}}\cdot 2\sum_{j \le h/2}j^{(h-1)k}\int_{0}^{1} \left(\sum_{i=0}^{j}(-1)^{i}\ncr{h}{i}\left(1-\frac{(i-t)}{j}\right)^{h-1}\right)^{k} dt.
\end{equation} These constants emerge in our phase transition function
\begin{equation}\label{eq:defnofgfunction}
g(x;s,d) \ := \ \sum_{k=1}^{\infty}(-1)^{k-1} \frac{b_{h,k}}{(s!d!)^k} x^{(s+d)k},
\end{equation} which for $h=s+d \ge 2$ converges for all $x$.

Our main result is the following.

\begin{thm}\label{thm:mainfastcriticaldecay} Let $h$ be a positive integer, $c > 0$ a real number, and choose pairs of integers $(s_i, d_i)$ with $s_i \ge d_i$ and $s_i+d_i = h$; for definiteness let $d_1 > d_2$. Consider subsets $A \subset I_N$ where each element of $I_N$ is independently chosen to be in $A$ with probability $p(N) = c N^{-\delta}$.

\begin{itemize}

\item For $\delta>\frac{h-1}{h}$, the set $A_{s_{i},d_{i}}$ with the larger $d_{i}$ is larger almost surely. In particular, as $N\to\infty$ with probability one we have $|A_{s_1,d_1}|/|A_{s_2,d_2}| = (s_2!d_2!) / (s_1!d_1!) + o(1)$.

\item If $\delta=\frac{h-1}{h}$ then almost surely $|A_{s_i,d_i}| \sim Ng(c;s_{i},d_{i})$ (with $g$ defined in \eqref{eq:defnofgfunction}), and thus with probability one $|A_{s_1,d_1}|/|A_{s_2,d_2}|$ is $g(c;s_1,d_1)/g(c;s_2,d_2) + o(1)$.

\end{itemize}
\end{thm}

Thus for two sets with the total number of summands fixed, the set with more minus signs is larger almost surely when $\delta > \frac{h-1}{h}$, so there are more distinct
elements in the generalized sumset with more minus signs. There is a phase transition in the behavior when $\delta$ passes from being greater than $\frac{h-1}{h}$ to equaling $\frac{h-1}{h}$.

The proof is similar to that in \cite{HM}, which does the $h=2$ case. The idea is to bound the number of times distinct $h_{(s,d)}$-tuples generate the same element. This allows us to discount the number of repeated elements in the generalized sumset. If we already know that most elements are distinct, then simple combinatorics allows us to compare their sizes; however, as $\delta$ gets smaller, we choose more and more elements for $A$, which leads to more repeated elements in the generalized sumset. The analysis is significantly easier when there are fewer repeated generalized sums, as then the sizes of the two generalized sumsets are well separated. Specifically, in the case of fast decay, the analysis follows from Chebyshev's inequality. The case of critical decay is significantly more challenging and requires recent strong concentration results. We first must show that we can estimate the number of $h_{(s,d)}$-tuples with a constant sum by the number of $h_{(s,d)}$-tuples with $h$ distinct elements. We then show that if we partition our $h_{(s,d)}$-tuples into equivalence classes based on the number of other $h_{(s,d)}$-tuples with the same sum, the class of singletons is the largest, so most $h_{(s,d)}$-tuples generate a unique integer. We then define our function $g(x;s,d)$ in terms of $|A_{s,d}|$.

In Section \ref{sec:secR(n,s,d)} we define $R(n,s,d)$ to count possible values for $A_{s,d}$. In Section \ref{sec:genhegmilrandvars}
we bound the expected number of repeated elements, and in Section \ref{sec:Strong-Concentration} we show that the number of repeated elements is close to its expected value. In Section \ref{sec:Case-i,-Theorem} we study the case of fast decay. In Section \ref{sec:Case-ii,-Theorem} we study the case
of critical decay. We end with a discussion of future work.


\section{Strong Concentration}\label{sec:strongconcprelim}

In this section, we first derive formulas for quantities related to the number of $h_{(s,d)}$-tuples generating a given number. These results are key ingredients in the strong concentration analysis.

\subsection{Determining $R(n,s,d)$}\label{sec:secR(n,s,d)}

Our first step is to determine a tractable formula for $R(n,s,d)$, the number of $h_{(s,d)}$-tuples of integers drawn from $\{0, \dots, N\}$ that generate $n$.

\begin{lem}\label{lem:R(,n,s,d)} Let $n' := n+dN$. We have
\be R(n,s,d)\ = \ \sum_{i=0}^{\lfloor\frac{n'}{N}\rfloor-1}(-1)^{i}\ncr{h}{i}\ncr{n'-i(N+1)+h-1}{h-1}.\ee
\end{lem}

\begin{proof} We first assume $d=0$, so all signs are positive and $n'=n$. By the stars and bars / cookie problem, the number of ways to write $n$ as a sum of $h$ non-negative integers is $\ncr{n+h-1}{h-1}$. As it will be important later, it is worth noting that this treats $4+3+1$ and $3+4+1$ as two different representations. Also, note this is equivalent to solving $x_1+\cdots+x_h = n$ with each $x_i$ a non-negative integer. We desire each summand to be in $I_N := \{0,\dots,N\}$, and thus $\ncr{n+h-1}{h-1}$ may overcount. We remedy this by using inclusion-exclusion to remove representations with summands exceeding $N$.

We first remove all representations where at least one summand exceeds $N$. There are $\ncr{h}{1}$ ways to choose which summand this is. We write that summand as $x_j = y_j+N+1$, and write $x_j=y_j$ for the remaining summands. Thus the number of representations where summand $j$ exceeds $N$ and the other summands are at least zero is the number of solutions to $y_1 + \cdots + y_h = n-(N+1)$, which is just $\ncr{n-(N+1)+h-1}{h-1}$. If instead $i$ summands are greater than $N$, we would get $y_1 + \cdots + y_h = n - i(N+1)$, for $\ncr{n-i(N+1)+h-1}{h-1}$ solutions. The claim now follows by inclusion-exclusion.

We only need trivial modifications if $d>0$. For the $d$ elements occurring with a minus sign, $a_{s+1}, \dots, a_{s+d}$, write $a_j' = N-a_j$. Then \be a_1 + \cdots + a_s - a_{s+1} - \cdots - a_{s_j} \ = \ n \ee becomes \be a_1 + \cdots + a_s + a_{s+1}' + \cdots + a_{s+d}' \ = \ n + dN, \ee reducing us to the first case.
\end{proof}

In our strong concentration applications we need not $R(n,s,d)$, but the closely related quantity $R_{{\rm distinct}}(n,s,d)$, which counts the number of representations of $n$ by $h=s+d$ distinct elements. The next lemma shows that these two quantities differ in a lower order term (relative to $N$).

\begin{lem}\label{lemmarepeats} The number of $h_{(s,d)}-tuples$ which generate
$n$ using $h$ distinct elements is of a higher order than repeated
elements. In particular, if $n' = n + dN$ then \be R(n,s,d) \ = \ R_{{\rm distinct}}(n,s,d) + O(N^{h-2}). \ee  \end{lem}

\begin{rek} If $(n')^{h-2} = o(N)$, the error term in Lemma \ref{lemmarepeats} exceeds the main term. While a more careful analysis gives a better error estimate, the bound above suffices for our applications as the main term is summed over a large enough regime that its contribution exceeds that of the error. \end{rek}

\begin{proof} If there is at least one repeated element, there are at most $h-2$ free choices for the summands (we lose one choice for the repetition, and one choice as the sum must equal $n$). Thus the contribution from representations of $n$ with a repeated element is at most $O(N^{h-2})$. \end{proof}


\subsection{Generalizing Hegarty-Miller's Random Variables}\label{sec:genhegmilrandvars}

Hegarty and Miller \cite{HM} introduce some useful random variables to prove their strong concentration results. We begin with a generalization of these quantities, and then derive useful bounds which give the needed asymptotic relations.

For a set $A$, define \begin{eqnarray}
 & A_{k} \ := \ \bigg\{\left\{\left\{ a_{1},...,a_{h}\right\}, \ldots, \left\{ a_{(k-1)h+1},...,a_{kh}\right\}\right\}: \sum_{i=1}^{s}a_{i}-\sum_{i=s+1}^{h}a_{i}\ = \ \cdots \nonumber\\ & \ \ \ \ \ \ \ \ \ \ = \ \sum_{i=(k-1)h+1}^{kh-d}a_{i}-\sum_{i=kh-d}^{kh}a_{i}\bigg\},
\end{eqnarray} and let $X_{k}=|A_{k}|$. Note that now the ordering of elements \emph{within} the $h$-tuples matters (because subtraction is not commutative), so we are looking at unordered $k$-tuples of \emph{ordered} elements. The dependence on the ordering is, however, weak. Given any one of these $k$-tuples, we can permute the first $s$ elements or permute the last $d$ elements without changing the number it generates, and thus such a permutation is the same element of the $k$-tuple. If all the elements of an $h$-tuple are distinct (actually, all we need are no repeats among the first $s$ and no repeats among the final $d$), then there are $s!d!$ ways to reorder the tuple \emph{without} changing the number it generates, and thus all of these correspond to the same \emph{set} (remember, the only way the ordering matters in the set of $h$ elements is which are the first $s$ elements and which are the last $d$). Thus, if all elements are distinct, there is overcounting by a factor of $s!d!$; we must take this into account later.

We want to study these $k$-tuples because they shed light on how
many repeated elements are in the generalized sumset. We have $k$-tuples of $h_{(s,d)}$-tuples, so each $k$-tuple has
a total of $hk$ integers. We place $h_{(s,d)}$-tuples in the same
$k$-tuple if they all generate the same number. Intuitively, because we need to subtract out repeated elements, all $h_{(s,d)}$-tuples within the same $k$-tuple only count once
in our generalized sumset, so counting these $k$-tuples is equivalent
to counting $h_{(s,d)}$-tuples.
To make this more concrete, we present a short example.  If $h=3$,
$s=3$, and $d=0$, then $\left\{ 3,4,7\right\} $, $\left\{ 5,6,3\right\} $,
$\left\{ 1,11,2\right\} $, and $\left\{ 1,5,8\right\} $ would all
be in the same $k$-tuple because they all sum to 14. If these four
$h_{(s,d)}$-tuples were the only $h_{(s,d)}$-tuples that generated
$14$, then we would have $A_{4}=\left\{ \left\{ 3,4,7\right\} ,\left\{ 5,6,3\right\} ,\left\{ 1,11,2\right\} ,\left\{ 1,5,8\right\} \right\} $.
$X_{k}$ counts the number of $k$-tuples, so the number of times
there are exactly $k$ $h_{(s,d)}$-tuples generating the same number.
For example, if we also only had $4$ $h_{(s,d)}$-tuples that generated
$5$, and $5$ and $14$ were the only two numbers generated, then
$X_{4}=2$ for the two different numbers generated by exactly $4$
$h_{(s,d)}$-tuples.

The reason why $A_{1}$ is so important is that $h_{(s,d)}$-tuples
are only in $A_{1}$ if no other $h_{(s,d)}$-tuples generate that
number. We want the number of single $h_{(s,d)}$-tuples
in $A_{1}$ (counted by $X_{1}$) to be the largest because then we
know that most $h_{(s,d)}$-tuples generate a unique sum. The larger
$k$ is, the more constant sums we must have. We have $k$-tuples of $h_{(s,d)}$-tuples, and within each $k$-tuple, all $h_{(s,d)}$-tuples contained inside generate the same number.
 If there were another $h_{(s,d)}$-tuple that generated the
same number, then the two $h_{(s,d)}$-tuples would be in $A_{2}$.
Therefore, $X_{1}$ counts the number of \textit{distinct} sums among
our $h_{(s,d)}$-tuples, which is important because we will show that
this is the higher order than $X_{k}$ for any $k>1$, so $X_{1}$
becomes critically important in measuring the size of the generalized
sumset.
So, $X_{1}$ counts the number of $h_{(s,d)}$-tuples that generate a
distinct sum, because if an $h_{(s,d)}$-tuple is in $A_{1}$, then
there are no other $h_{(s,d)}$-tuples that generate the same number.
Similarly, $X_{2}$ counts how many $h_{(s,d)}$-tuples generate the
same sum as exactly one other $h_{(s,d)}$-tuple.
Therefore, if we know
that $X_{2}=o(X_{1})$, then we know there are significantly more $h_{(s,d)}$-tuples
with a unique sum than those with any number of repeated sums (because
any $k$-tuples in $A_{k}$ for $k>\ell$ are also in $A_{\ell}$).

The goal is to generalize Theorem 1.1 and Lemma 2.1 of \cite{HM}. To do this we must bound the number of repeated elements in $A_{s,d}$. \emph{If} we knew that our generalized sumset contains mostly distinct sums (so most $h_{(s,d)}$-tuples generate a distinct integer), then
a simple combinatorial argument and Chebyshev's theorem would suffice to prove Theorem \ref{thm:mainfastcriticaldecay}. In the case of fast decay, $\delta>\frac{h-1}{h}$, the number of repeated elements is a lower order than the number of distinct elements. The case of critical
decay, $\delta=\frac{h-1}{h}$, is more difficult because now the number of our repeated elements is of the same order as the number of distinct elements. Intuitively, the smaller $\delta$ is, the more elements from $\{0, \dots, N\}$ are in $A$, so the more likely it is
that two $h_{(s,d)}$-tuples generate the same element. Thus a more sophisticated argument is needed to find the relevant cardinalities.

We first introduce some terminology.

\begin{definition} By \textbf{Type $0$} we mean the $k$-tuples with $hk$ distinct elements of $I_{N}$, while \textbf{Type $i$} refers to $k$-tuples with $i$ repeated elements. \end{definition}

By repeated elements, we mean total number of elements that would
need to be removed for all elements to be distinct. For example, in
the $7$-tuple $\left\{ 1,1,1,2,2,3,4\right\} $, we say there are
three repeated elements because we would need to remove $\left\{ 1,1,2\right\} $
for all remaining elements to be distinct. For a fixed $k$-tuple $\alpha$, since we draw our $A$ from a binomial model with parameter $p(N) = c N^{-\delta}$, we know
\begin{equation}\label{eq:typet}
{\rm Prob}(\alpha {\rm\ is\ of\ Type}\ t)\ =\ncr{hk}t \ c^{kh-t}N^{-\delta(kh-t)}.
\end{equation}
Equation \eqref{eq:typet} holds because the probability of choosing
any element is independent of the probability of choosing any other
element. We need a binomial coefficient because we have to choose $t$ of the $k$-tuple's total $hk$ elements to repeat. Note that \eqref{eq:typet} is for a fixed $k$-tuple, but we do not know the locations of the repeated elements, so the binomial coefficient is necessary for all possible combinations of repeats.

Let $\xi_{i,k}(N)$ be the number of $k$-tuples of type $i$. Note that we have $k$-tuples of $h$-element sets; in those $h$ element sets, the ordering of elements within matters a bit, though we may permute the first $s$ or permute the last $d$ without changing the number it generates. As in equation (2.4) of \cite{HM},
\begin{equation}
\E(X_{k})\ = \ \sum_{i=0}^{hk-1}\xi_{i,k}(N)p(N)^{(k-i)h}.\label{eq:1}
\end{equation}
This holds because we are summing over all possible types of $k$-tuples times the probability of choosing a $k$-tuple of that type, so we get the expected number of $k$-tuples.

Similar to \cite{HM}, for $\delta \ge \frac{h-1}{h}$, the only contribution to \eqref{eq:1} that matters is from the first term. This is equivalent to estimating the number of $k$-tuples by only considering the number of $k$-tuples with no repeated elements. We first estimate the contribution from this term, and then bound the contribution from the remaining ones.


\begin{lem}\label{prop:e_0} We have
\begin{equation}
\xi_{0,k}(N)\ \sim\ \frac{b_{h,k}}{(s!d!)^k} N^{(h-1)k+1}.
\end{equation}
The error above is $O(N^{(h-2)(k-1)+1})$, and $b_{h,k}$ is defined in
\eqref{eq:eqnB}.
\end{lem}



\begin{proof} Because we have to sum over all $n$ in the interval to count how many times a $k$-tuple can generate the same number, the number of $k$-tuples of Type $0$ is
\begin{equation}
\xi_{0,k}(N)\ = \ \sum_{n=-dN}^{sN}\ncr{R(n,s,d) / s!d! + O(N^{h-2})}{k},\label{eq:e_onesum}
\end{equation} where $R(n,s,d)$ is the number of $h_{(s,d)}$-tuples elements in $\{0,\dots,N\}$ that generate $n$. The error is because $\xi_{0,k}(N)$ counts \emph{distinct} tuples, while $R(n,s,d)$ allows repeats; however, our earlier analysis showed that the number of tuples with repeated elements is lower order (this is because $h$ is fixed and $N$ tends to infinity). From the Binomial Theorem and standard bounds on approximating binomial coefficients with the largest term (specifically, $\ncr{f(n)}m=\frac{f(n)^{m}}{m!}+O(f(n)^{m-1})$), we find
\be\label{epzero} \xi_{0,k}(N)\ = \ \sum_{n=-dN}^{sN}\ncr{R(n,s,d) / s!d!}{k} + O\left(\sum_{n=-dN}^{sN}\ncr{N^{h-2}}{k}\right) \ \sim \ \sum_{n=-dN}^{sN}\ncr{R(n,s,d) / s!d!}{k}.
\end{equation}

Letting $n' = n + dN$ as before, define
\begin{equation}
S_{j}(N) \ :=\
\sum_{n'=jN}^{(j+1)N}\ncr{R(n',s,d)/s!d!}{k}.\label{eq:S_j}
\end{equation}

We use the notation $S_{j}(N)$ to sum over all possible $n$ in $R(n,s,d)$ in one of $h$ intervals of length $n$. In $S_{j}(N)$,
$j$ gives the index of the interval. From \eqref{epzero}, we see that it is useful to break $\xi_{0,k}(N)$ into these intervals
in order to compute the total sum. To distinguish between $R(n,s,d)$ and $S_{j}(N)$, recall that $R(n,s,d)$ is for a fixed $n$, while
$S_{j}(N)$ is for a fixed interval of length $N$.

Assume $h$ is even (the case of $h$ odd is similar). We first approximate $R(n,s,d)$. Let $j = \lfloor\frac{n'}{N}\rfloor-1$. Assume $j>0$; the case of $j= 0$ follows similarly, and we mostly omit the details. We have
\begin{eqnarray}\label{eq:R(n)approx}
R(n',s,d) & \  = \ & \sum_{i=0}^{\lfloor\frac{n'}{N}\rfloor-1}(-1)^{i}\ncr{h}{i}\ncr{n'-i(N+1)+h-1}{h-1} \nonumber\\ &
\ =\ & \sum_{i=0}^{j}(-1)^{i}\ncr hi\frac{(n'-iN)^{h-1}}{(h-1)!} + O(N^{h-2});
\end{eqnarray} this follows from standard approximation for the binomial coefficients and the Binomial Theorem.\\

\noindent \textbf{For the rest of this subsection, in all the analysis below the error term in the asymptotic relations denoted by $\sim$ are at least one order smaller in $N$.} \\

We now have
\begin{eqnarray}
  S_{j}(N) & \ = \ & \sum_{n'=jN}^{(j+1)N}\ncr{\frac{1}{s!d!}\sum_{i=0}^{j}\left(-1\right)^{i}\ncr hi\frac{\left(n'-iN\right)^{h-1}}{(h-1)!}}k+O\left(\sum_{n'=jN}^{(j+1)N}\right)\left(n-iN\right)^{h-2}\nonumber \\
 &  \sim & \sum_{n'=jN}^{(j+1)N}\frac{\left(\frac{1}{s!d!}\right)^{k}\left(\sum_{i=0}^{j}\left(-1\right)^{i}\ncr hi\left(\frac{\left(n'-iN\right)^{h-1}}{(h-1)!}\right)^{k}\right)}{k!}
\end{eqnarray}


We pull out all terms that do not depend on $n'$ to get
\begin{equation}
S_{j}(N)\ \sim\ \frac{1}{(s!d!)^{k}k!(h-1)!^{k}}\sum_{n'=jN}^{(j+1)N}\left(\sum_{i=0}^{j}(-1)^{i}\ncr hi(n'-iN)^{h-1}\right)^{k}.\label{eq:S_jreduced}
\end{equation}
Our goal is to find the dependence on $s, d$ and $N$. To do this, we first approximate \eqref{eq:S_jreduced} with an integral
to get
\begin{equation}
S_{j}(N)\ \sim\ \frac{1}{(s!d!)^{k}k!(h-1)!^{k}}\int_{jN}^{(j+1)N}\left(\sum_{i=0}^{j}(-1)^{i}\ncr hi(x-iN)^{h-1}\right)^{k}dx;
\end{equation} the cost of the approximation is one order lower in $N$ as we have sums of polynomials.

We change variables by taking $x=(j+t)N$ with $t$ ranging from 0 to 1. Thus $dx=Ndt$ and as $j > 0$ (if $j=0$ we cannot pull out the power of $j$)
%
\begin{eqnarray}
S_{j}(N) & \ \sim\  & \frac{1}{(s!d!)^{k}k!(h-1)!^{k}}j^{(h-1)k}N^{k(h-1)+1}\int_{0}^{1}\left(\sum_{i=0}^{j}(-1)^{i}\ncr hi\left(1-\frac{(i-t)}{j}\right)^{h-1}\right)^{k}dt.\nonumber\\ 
\end{eqnarray} From our definition of $b_{h,k}$ (see \eqref{eq:eqnB}) and the fact that by symmetry it suffices to sum $n'$ up to $hN/2$, summing over $j$ completes the proof. \end{proof}

%
%

\begin{lem}\label{thm:EX_k,} We have
\begin{equation}
\E(X_{k})\ \sim\ \xi_{0,k}(N)p(N)^{-hk\delta} \ = \ \frac{b_{h,k} c^{hk}}{(s!d!)^k} N^{(h-1)k+1-hk\delta},
\end{equation} with $b_{h,k}$ defined in \eqref{eq:eqnB}.
\end{lem}

\begin{proof} By Lemma \ref{prop:e_0}, it suffices to show $\E(X_{k}) \sim \xi_{0,k}(N)p(N)^{-hk\delta}$.
To show that we can estimate $\E(X_{k})$ by $\xi_{0,k}(N)p(N)^{hk}$,
it suffices to prove that for each $\ell > 0$ we have
\begin{equation}
\xi_{\ell,k}(N)p(N)^{hk-\ell} \ \ = \  \
o(\xi_{0,k}(N)p(N)^{hk}).
\end{equation}

From \eqref{eq:typet}, for $\ell>0$ the probability of being Type $\ell$ is $p(N)^{hk-\ell}$ and the number of such $k$-tuples
is $\xi_{\ell,k}(N)$. The repeated elements can either be in the same $h_{(s,d)}$-tuple or in different $h_{(s,d)}$-tuples. In both cases we have the same order, though. We have $k$ sets of $h$-tuples. That would give us $hk$ independent variables; however, each of the $h$-tuples must sum to $N$ (so we lose $k$ degrees of freedom), and then we lose another $\ell$ by assumption (if $\ell=0$ we have no repeated elements, which is the main term). Thus for a fixed $n$ the number of solutions is at most on the order of $N^{hk-k-\ell}$; summing over $n$ gives at most order $N$, for a total contribution of at most order $N^{(h-1)k-\ell+1}$.

We now multiply by the probability $p(N)^{hk-\ell}$ and get
\begin{equation}
\xi_{\ell,k}(N)p(N)^{hk-\ell}\ = \ O\left(N^{(h-1)k-\ell+1-(hk-\ell)\delta}\right).
\end{equation}
Because $\delta<1$ and $\ell > 0$, we know that
\begin{equation}
\xi_{\ell,k}(N)p(N)^{hk-\ell} \ = \ O\left(N^{(h-1)k+1- hk\delta - \ell(1-\delta)}\right)\ = \ O\left(\xi_{0,k}(N)N^{-\delta hk} \cdot N^{-\ell(1-\delta)}\right),
\end{equation}
so the probability of choosing $k$-tuples with $\ell$ repeats is of a lower order than the probability of choosing a $k$-tuple
with no repeats, completing the proof.
\end{proof}

%

\subsection{Strong Concentration Results\label{sec:Strong-Concentration}}

We need to show $X_{k}$ is strongly concentrated about its expected
value as $N\rightarrow\infty$ to conclude that
the actual number of distinct elements in the generalized sumset approaches
the expectation. We know from Lemma \ref{thm:EX_k,} that the $expected$
number of distinct elements is of a higher order than the $expected$
number of repeated elements, but if we do not know that the $actual$
number of distinct elements is close to its expectation, then Lemma
\ref{thm:EX_k,} is of little use. Here we show that the actual number
does indeed approach its mean. This is similar to equations (2.9) and (2.10)
of \cite{HM}.

\begin{lem}
\label{lem:strong_concentration}For \textup{$\delta\ge\frac{h-1}{h}$}, $X_{k}$ becomes strongly concentrated about its expected value
as $N\rightarrow\infty$.\end{lem}
\begin{proof}
We employ a second moment method to show that $N^{-\big(\frac{(h-1)k+1}{kh}\big)}=o(p(N))$
implies $X_{k}$ is highly concentrated about its mean.

Let $\triangle=\sum_{a\sim b}P(Y_{\alpha}\cap Y_{\beta})$ where
$\alpha\sim\beta$ if $k$-tuples $\alpha,\beta$ have at least one number in
common, and $Y_{\alpha}$ is an indicator variable for each unordered
$k$-tuple having a constant sum. As $N\rightarrow\infty$, from Lemma
\ref{thm:EX_k,} we know that $\E(X_{k})\rightarrow\infty$ so, as
in equation (2.9) of \cite{HM}, it suffices to show that
\begin{equation}
\triangle\ = \ o(\E(X_{k})^{2})\ = \ o_{k}\left(\left(N^{2(h-1)k+2}\right)\left(c^{2hk}N^{-2hk\delta}\right)\right).
\end{equation}

The main contribution is from pairs
with $hk$ distinct elements and exactly $1$ element in common. From
Proposition \ref{prop:e_0}, we have $O(N^{(h-1)k+1})$ choices for
$\alpha$. There are $hk$ choices for common element with $\beta$,
$O(N^{(h-1)k})$ choices for the rest of $\beta$, and $2kh-1$ elements
in $\alpha\cup\beta$, so
\begin{equation}
P(Y_{\alpha}\cap Y_{\beta})\ = \ O(p(N)^{2kh-1}).
\end{equation}

Generalizing equation (2.10) in \cite{HM},
\begin{eqnarray}
 \triangle & \ = \ & \sum_{a\sim b}P(Y_{\alpha}\cap Y_{\beta})\nonumber \\
& \ = \ &\sum_{a\sim b}p(N)^{2kh-1}\nonumber \\
& \ = \ &O(N^{(h-1)k+1+(h-1)k})c^{2kh-1}N^{-(2kh-1)\delta})\nonumber \\
& \ = \ & O(N^{2k(h-1)+1-(2kh-1)\delta}).
\end{eqnarray}

Because $\delta<1$,
\begin{equation}
\triangle\ = \ o_{k}(N^{2k(h-1)+2-2kh\delta}).\label{eq:delta}
\end{equation}

which proves our lemma.$ $
\end{proof}

\section{Phase Transition}

\subsection{Fast Decay\label{sec:Case-i,-Theorem}}

Here we prove the first claim of Theorem \ref{thm:mainfastcriticaldecay}. We can do this using
Chebyshev's inequality and Lemmas \ref{thm:EX_k,} and \ref{lem:strong_concentration}.
This is equations 2.11-2.12 of \cite{HM}.
For $\delta>\frac{h-1}{h}$,
\begin{align}
 & X_{1}\ \sim\ E(X_{1})\ \sim\ \left(b_{h,1}N^{h}\right)\left(c^{h}N^{-h\delta}\right)\nonumber \\
 & X_{2}\ \sim\ \left(b_{h,2}N^{2(h-1)+1}\right)\left(c^{2h}N^{-2h\delta}\right).
\end{align}

We get the above equations from plugging $k=1,2$ into Lemma \ref{thm:EX_k,}.
Because $\delta>\frac{h-1}{h}$ and $X_{1}=\Theta(N^{h-h\delta})$,
\begin{equation}
X_{2}\ = \ O(N^{2(h-1)+1-2h\delta})\ = \ O(X_{1})+O(N^{-h+1+h\delta}).
\end{equation}

The error term is lower order, so as $N\rightarrow\infty$, all but
a vanishing proportion of $h_{(s,d)}$-tuples will generate distinct
sums.

\subsection{Critical Decay\label{sec:Case-ii,-Theorem}}

We are now ready to prove  the second claim in Theorem \ref{thm:mainfastcriticaldecay}. The key result in our earlier approximation of $X_{k}$ is that
when we plug in $\delta\ = \ \frac{h-1}{h}$, all the exponents on $N$
sum to $1$, so we are left with a term on the order of $N$.

We first claim
\begin{equation}
\left||A_{s,d}|-\sum_{k=1}^{m}(-1)^{k-1}X_{k}\right|\ \leq\ X_{m}.\label{eq:onesum}
\end{equation}

We omit the proof because of its similarity to \cite{HM} (see equation (2.16) there).

We now want to show
\begin{equation}
|A_{s,d}|\ \sim\ \sum_{k=1}^{m}(-1)^{k-1}X_{k}.
\end{equation}

To do this, we need to show that the coefficients on $X_{m}$ go to
$0$ as $k\rightarrow\infty$. The proof of this is a rote bound; we omit the details. By our concentration result
in Lemma \ref{lem:strong_concentration},
\begin{equation}
X_{m}\ \sim\ E(X_{m})\ \sim\ b_{h,k}N^{(h-1)k+1-hk\delta}c^{hk}\ \sim\ b_{h,k}c^{hk}N.\label{eq:coeffics0}
\end{equation}

Therefore, because $\delta=\frac{h-1}{h}$,
\begin{equation}
X_{m}\ \sim\ b_{h,k}c^{hk}N.
\end{equation}
Following equation (2.18) of \cite{HM} and
using equation (\ref{eq:onesum}):
\begin{equation}
|A_{s,d}|\ \sim\ \sum_{k=1}^{m}(-1)^{k-1}X_{k}\ \sim\ N\sum_{k=1}^{m}(-1)^{k-1}b_{h,k}c^{hk}.\label{eq:sizegeneralized}
\end{equation}

We conclude that
\begin{equation}
S_{d}^{s}\ \sim\ Ng(c;s,d).
\end{equation}

We define $g(c;s,d)$ to capture the $N$-dependency
of the size of our generalized sumset $A_{s,d,}$. Unlike in the case
of $A+A$ versus $A-A$, this function no longer has a nice closed
form. The function we have defined arises in the generalization of
Hegarty-Miller's random variables, and the purpose of $g(c;s,d)$
is to identify and pull out the N to determine how the size of the
generalized sumset depends on $N$.

We want to compare the sizes of two sets $A_{s_{1},d_{1}}$
and $A_{s_{2},d_{2}}$ for $s_{1}+d_{1}=s_{2}+d_{2}=h.$ The $k,h,N$
factors are all the same and cancel, so $|A_{s_1,d_1}|/|A_{s_2,d_2}|$
depends only on $s_{1}$, $s_{2}$, $d_{1}$, $d_{2}$. Therefore,
\begin{equation}
\frac{|A_{s_1,d_1}|}{|A_{s_2,d_2}|} \ = \ \frac{s_{2}!d_{2}!}{s_{1}!d_{1}!}.
\end{equation}  Because $d\leq s$, the maximum value of $1/(s!d!)$ is achieved
at the minimum value of $s!d!$, which occurs when $s=d$. Thus, we
conclude that as $N\rightarrow\infty$, with a probability of choosing
elements decaying in $N$, the set with the most minus signs is almost
surely larger. This proves the second claim of Theorem \ref{thm:mainfastcriticaldecay}.

\subsection{Future Work: Slow Decay}\label{sec:slowdecay}

We are left with the case when $\delta<\frac{h-1}{h}$. This was done in the third case of Theorem 1.1 in \cite{HM} for two summands, but in the general case of slow decay it is considerably more difficult for a number of reasons. The crucial difference in the analysis of the case of critical decay and the case of slow decay is that the case of slow decay focuses on the number of elements missing from $A_{s,d}$, while the case of critical decay focuses on the number of elements present in $A_{s,d}$. In the previous sections, we approximated $|A_{s,d}|$ by focusing on the middle of the interval $[-dN, sN]$ because it was here that elements were most likely to be present. However, to measure the number of sums missing from $A_{s,d}$, we instead need to look at the fringes of the interval, so the analysis shifts completely.
Following \cite{HM}, we would need to estimate the expectation of the number of elements missing from the generalized sumset.
In \cite{HM}, they let $\mathscr{E}_{n}$ denote the event that $n \not\in A+A$. They can then find the expected number of missing sums,
\begin{equation} \mathbb{E}[\mathscr{S}^{c}] \ = \  \sum_{n=0}^{2N}
\mathbb{P}(\mathscr{E}_{n});  \end{equation}
however, to find ${P}(\mathscr{E}_{n})$, they use that all ways of representing any integer $n$ are independent of one another. This leads to the following nice equation in \cite{HM}:
\begin{equation}  \twocase{
\mathbb{P}(\mathscr{E}_{n}) \ = \ }{(1-p^{2})^{n/2} (1-p)}{if $n$ is
even}{(1-p^{2})^{(n+1)/2}}{if $n$ is odd.} \end{equation}
In the general case, this formula is significantly less tractable as now the various ways to summing to $n$ all depend on one another. The probability must be conditioned on each previous element chosen to be in the $h_{s,d}$-tuple, and that is the major difficulty in finding this formula in the general case.
In the next equation of \cite{HM}, they sum over the probabilities of each $n$ in the interval:
\begin{equation} \mathbb{E}[\mathscr{S}^{c}]\ \sim\ 4 \cdot
\sum_{m=0}^{\lfloor N/2 \rfloor} (1-p^{2})^{m}\ \sim\ {4 \over p^{2}}; \end{equation}
however, in the $h=2$ case, this summation takes advantage of nice geometric series properties which are not have available in the general case, and are thus left for future work.



\ \\

\end{document}